\documentclass[12pt,oneside]{amsart}

\usepackage{amssymb}
\usepackage{amsmath}
\usepackage{amsthm}

\usepackage{longtable}

\theoremstyle{plain}
\newtheorem{theorem}{Theorem}
\newtheorem{lemma}{Lemma}
\newtheorem{corollary}{Corollary}
\newtheorem{proposition}{Proposition}

\theoremstyle{definition}

\theoremstyle{remark}
\newtheorem{remark}{Remark}

\begin{document}
\title
[The extension of holomorphic functions on a non-pluriharmonic locus]
{The extension of holomorphic functions on a non-pluriharmonic locus}
\author{Yusaku Tiba}
\date{}

\begin{abstract}
Let $n \geq 4$ and let $\Omega$ be a bounded hyperconvex domain in $\mathbb{C}^{n}$.  
Let $\varphi$ be a negative exhaustive smooth plurisubharmonic function on $\Omega$.  
We show that any holomorphic function defined on a connected open neighborhood of the support of $(i\partial \overline{\partial}\varphi)^{n-3}$ can be extended to the holomorphic function
on $\Omega$.  
\end{abstract}
\maketitle


\section{Introduction}
Hartogs extension theorem is stated as follows: 

{\em Let $\Omega$ be an open subset in $\mathbb{C}^{n}$ ($n \geq 2$) and let $K \subset \Omega$ be a compact subset such that $\Omega \setminus K$ is connected.  
Then any holomorphic function on $\Omega \setminus K$ can be extended to a holomorphic function on $\Omega$.}

This is one of the major difference between the theory of one and several complex variables since any open subset is a domain of holomorphy in the case of one variable.  
In this paper, we give a new example of a subdomain such that any holomorphic function on the subdomain can be extended holomorphically to the entire domain.  

Let $T$ be a smooth form or a current in a domain in $\mathbb{C}^{n}$.  
We denote by $\mathrm{supp}\, T$ the support of $T$.  
Our main theorem is the following: 
\begin{theorem}\label{theorem:1}
Let $n \geq 4$ and $\Omega$ be a bounded hyperconvex domain in $\mathbb{C}^{n}$.  
Let $\varphi$ be a negative smooth plurisubharmonic function on $\Omega$ such that $\varphi(z) \to 0$ when $z \to \partial \Omega$.  
Let $V \subset \Omega$ be a connected open neighborhood of $\mathrm{supp}\, (i\partial \overline{\partial}\varphi)^{n-3}$.  
Then any holomorphic function on $V$ can be extended to the holomorphic function on $\Omega$.  
\end{theorem}
If a holomorphic function is defined on a non-pluriharmonic locus, we can remove the assumption of the regularity of $\varphi$.  
\begin{corollary}\label{corollary:1}
Let $n \geq 4$ and $\Omega$ be a bounded hyperconvex domain in $\mathbb{C}^{n}$.  
Let $\varphi$ be a negative continuous plurisubharmonic function on $\Omega$ such that $\varphi(z) \to 0$ when $z \to \partial \Omega$.  
Let $V \subset \Omega$ be a connected open neighborhood of $\mathrm{supp}\, i\partial \overline{\partial}\varphi$.  
Then any holomorphic function on $V$ can be extended to the holomorphic function on $\Omega$.  
\end{corollary}
We explain a motivation of Theorem~\ref{theorem:1}.  
Let $E$ be a compact subset of $\mathbb{C}^{n}$.  
We define the Shilov boundary of $E$ by the smallest closed subset $\partial_{S}E$
of $E$ such that, for each function $f$ which is holomorphic on a neighborhood of 
$E$ the equality $\max_{E}|f| = \max_{\partial_{S}E}|f|$ holds.  
Let $B_{\varphi}(r)=\{z \in \Omega \, |\, \varphi(z) < r\}$ and let $x \in B_{\varphi}(r)$.  
It is known that there exists a probability measure $\mu_{x}$ supported on $\partial_{S}\overline{B_{\varphi}(r)}$ such that 
$f(x) = \int f d\mu_{x}$ for any holomorphic functions on an open neighborhood of $\overline{B_{\varphi}(r)}$ (see \cite{Duv-Sib}).  
This measure is called Jensen measure.  
Hence we may consider that Shilov boundaries of $\overline{B_{\varphi}(r)}$ ($r < 0$) 
are important for the existence of holomorphic functions on $\Omega$.  
On the other hand, \cite{Bed-Kal} shows that there exists a complex foliation on $\Omega \setminus \mathrm{supp}\, (i\partial \overline{\partial} \varphi)^{j}$ ($1 \leq j \leq n$) by complex submanifolds such that the restriction of $\varphi$ on any leaf of the foliation is pluriharmonic.  
It follows that, for any $z \in \Omega \setminus \mathrm{supp}\, (i\partial \overline{\partial}\varphi)^{n-1}$, there exists a complex curve through $z$ contained in a level set of $\varphi$.  
Then $z$ is not contained in the Shilov boundaries of level sets of $\varphi$.  
In this context, it might be interesting to ask whether one can extend holomorphic functions defined on $\mathrm{supp}\, (i\partial \overline{\partial} \varphi)^{n-1}$ to the holomorphic functions on $\Omega$.  
In our theorem, we show that this question is true if $\mathrm{supp}\, (i\partial \overline{\partial}\varphi)^{n-1}$ is replaced by $\mathrm{supp}\, (i\partial \overline{\partial}\varphi)^{n-3}$.  

The proof consists in solving $\overline{\partial}$ equation in the $L^{2}$-space defined by the degenerate Monge-Amp\`ere measure.  
In Section~\ref{section:3}, we prove Donnelly-Fefferman and Berndtsson type $L^{2}$-estimate (\cite{Don-Fef}, \cite{Ber}) .  
In Lemma~\ref{lemma:1} and Lemma~\ref{lemma:3}, we use the argument in Theorem~2.3 of \cite{Ber-Cha} to prove our $L^{2}$-estimate from (\ref{equation:2}) below.  
We solve $\overline{\partial}$ equations in the $L^{2}$-spaces defined by the complete K\"ahler metrics which converge to $-i\partial \overline{\partial}(\log (-\varphi))$, which is no longer a K\"ahler metric in general.  
To guarantee the weak convergence of solutions constructed in Section~\ref{section:3}, 
we show an interior estimate of the solutions in Section~\ref{section:4}.  
Section~\ref{section:4} can be read independently of other sections.  
\medskip

{\it Acknowledgment.}
This work was supported by the 
Grant-in-Aid for Scientific Research (KAKENHI No.\! 17K14200).  

\section{Preliminaries}\label{section:2}
First we introduce the set up and some notation.  
For details, we refer to \cite{Dem}.  
Let $\Omega$ be a domain in $\mathbb{C}^{n}$ and let $\omega$ be a K\"ahler metric on $\Omega$.  
Let $\psi$ be a smooth function on $\Omega$.  
By $L^{2}_{p, q}(\Omega, e^{\psi}, \omega)$ we denote the Hilbert space of $(p, q)$-forms $\alpha$ which satisfy 
\[
\|\alpha\|_{\psi, \omega}^{2} := \int_{\Omega}|\alpha|_{\omega}^{2}e^{\psi}dV_{\omega}.  
\]
Here $dV_{\omega} = (n!)^{-1}\omega^{n}$.   
For simplicity we put $L^{2}(\Omega, e^{\psi}, \omega) = L^{2}_{0, 0}(\Omega, e^{\psi}, \omega)$.  
Let $A^{2}(\Omega, e^{\psi}, \omega)$ be the space of all holomorphic functions in $L^{2}(\Omega, e^{\psi}, \omega)$.  
Let $\overline{\partial}_{\psi}^{*}$ be the Hilbert space adjoint of linear, closed, densely defined operator 
\[
\overline{\partial} : L^{2}_{p, q}(\Omega, e^{\psi}, \omega) \to L^{2}_{p. q+1}(\Omega, e^{\psi}, \omega).  
\]
Let $\Lambda_{\omega}$ be the adjoint of multiplication by $\omega$.  
If $q \geq 1$ and $\omega$ is a complete K\"ahler metric, 
the Bochner-Kodaira-Nakano inequality shows that 
\[
\|\overline{\partial} \alpha\|_{\psi, \omega}^{2} + \|\overline{\partial}_{\psi}^{*} \alpha\|^{2}_{\psi, \omega} \geq \int_{\Omega} \langle [-i\partial \overline{\partial}\psi, \Lambda_{\omega}]\alpha, \alpha \rangle_{\omega} e^{\psi} dV_{\omega} 
\]
for any $\alpha \in L^{2}_{p, q}(\Omega, e^{\psi}, \omega)$ which is contained in the both domains of $\overline{\partial}$ and $\overline{\partial}^{*}_{\psi}$.  
At each point $x \in \Omega$, we may choose an orthonormal basis $\sigma_{1}, \ldots, \sigma_{n}$ 
for the holomorphic cotangent bundle with respect to $\omega$ such that $i \partial \overline{\partial} \psi = \lambda_{1} i \sigma_{1} \wedge \overline{\sigma_{1}} + \cdots + \lambda_{n} i \sigma_{n} \wedge \overline{\sigma_{n}}$.  
Let $\alpha$ be a $(0, q)$-form.  
We write 
$\alpha = \sum_{|J| = q}\alpha_{J} \overline{\sigma_{J}}$ where $J = (j_{1}, \ldots, j_{q})$ is a multi-index with $j_{1} < \cdots < j_{q}$ and $\overline{\sigma_{J}} = \overline{\sigma_{j_{1}}} \wedge \cdots \wedge \overline{\sigma_{j_{q}}}$.   
Then \begin{equation}\label{equation:1}
[-i\partial \overline{\partial}\psi, \Lambda_{\omega}]\alpha = \sum_{|J| = q} \left( \sum_{1 \leq j \leq n, j \not \in J} \lambda_{j} \right) \alpha_{J} \overline{\sigma_{J}}
\end{equation}
Assume that the operator $A_{\omega, \psi} = [-i\partial \overline{\partial}\psi, \Lambda_{\omega}]$ is positive definite on $\Omega$, and that $\omega$ is a complete K\"ahler metric.  
Then, for any closed form $\alpha \in L^{2}_{p, q}(\Omega, e^{\psi}, \omega)$ which satisfies 
$\int_{\Omega} \langle A_{\omega, \psi}^{-1}\alpha, \alpha \rangle_{\omega} e^{\psi} dV_{\omega} < +\infty$, there exists $u \in L^{2}_{p, q-1}(\Omega, e^{\psi}, \omega)$ such that $\overline{\partial}u = \alpha$ and 
\begin{equation}\label{equation:2}
\|u\|^{2}_{\psi, \omega} \leq \int_{\Omega} \langle A_{\omega, \psi}^{-1}\alpha, \alpha \rangle_{\omega} e^{\psi}dV_{\omega}.  
\end{equation}

\section{Weighted $L^{2}$-estimate}\label{section:3}
The purpose of this section is Proposition~\ref{proposition:1} below.  
Let $n \geq 4$.  
Let $\Omega \subset \mathbb{C}^{n}$ be a bounded hyperconvex domain and let $\varphi$ be a negative smooth plurisubharmonic function on $\Omega$ such that $\varphi(z) \to 0$ when $z \to \partial \Omega$.  
Let $\phi = - (\log(-\varphi))$.  
Then $\phi$ is an exhaustive smooth plurisubharmonic function such that $i \partial \overline{\partial} \phi \geq i \partial \phi \wedge \overline{\partial} \phi$.  
Let $\psi$ be a smooth strongly plurisubharmonic function on a neighborhood of $\overline{\Omega}$.  
To prove our main theorem, we may assume that $\psi = |z|^{2}$.  
Let $\varepsilon >0$ be a small positive number and let $\omega_{\varepsilon} = i\partial \overline{\partial}(\varepsilon \psi + \phi)$.  
Then $\omega_{\varepsilon}$ is a complete K\"ahler metric on $\Omega$ since $\phi$ is exhaustive and $|\partial \phi|_{\omega_{\varepsilon}} < 1$.  
Let $c > 0$ and let $A_{\varepsilon, c} = A_{\omega_{\varepsilon}, \psi + c \phi} = [-i\partial \overline{\partial}(\psi + c \phi), \Lambda_{\omega_{\varepsilon}}]$.  
We start by showing the following lemma: 
\begin{lemma}\label{lemma:1}
Let $\delta, \delta' >0$.   
Let $\alpha$ be a $\overline{\partial}$-closed $(0, 1)$-form such that 
$\alpha \in L^{2}_{0, 1}(\Omega, e^{\psi + \delta'\phi}, \omega_{\varepsilon})$.  
Assume that $n > 1 + \delta$ and that $\varepsilon \leq \delta^{-1}$.  
Then there exists a function $u \in L^{2}(\Omega, e^{\psi+\delta'\phi}, \omega_{\varepsilon})$ such that $\overline{\partial} u = \alpha$ and 
\[
\int_{\Omega}|u|^{2}e^{\psi-(\delta-\delta') \phi}dV_{\omega_{\varepsilon}} \leq C_{n, \delta}\int_{\Omega}\langle A_{\varepsilon, \delta+ \delta'}^{-1}\alpha, \alpha \rangle_{\omega_{\varepsilon}} e^{\psi -(\delta-\delta') \phi}dV_{\omega_{\varepsilon}}.  
\]
Here $C_{n, \delta}$ is a positive constant which depends only on $n$ and $\delta$.  
\end{lemma}

\begin{proof}
Since $C\omega_{\varepsilon} \leq i\partial \overline{\partial}(\psi + \delta'\phi)$ for some $C>0$, 
we have that $A^{-1}_{\varepsilon, \delta'} \leq C^{-1}$.   
By (\ref{equation:2}), there exists the solution $u \in L^{2}(\Omega, e^{\psi + \delta'\phi}, \omega_{\varepsilon})$ to $\overline{\partial} u = \alpha$ which is minimal in the $L^{2}(\Omega, e^{\psi+\delta'\phi}, \omega_{\varepsilon})$ norm.  
This means that $u \in A^{2}(\Omega, e^{\psi+\delta'\phi}, \omega_{\varepsilon})^{\perp}$.  
Since $\phi(z) \to +\infty$ when $z \to \partial\Omega$, 
we have that $ue^{-\delta \phi} \in L^{2}(\Omega, e^{\psi + (\delta + \delta') \phi}, \omega_{\varepsilon})$ and 
$A^{2}(\Omega, e^{\psi + (\delta + \delta') \phi}, \omega_{\varepsilon}) \subset A^{2}(\Omega, e^{\psi+\delta'\phi}, \omega_{\varepsilon})$.  
Hence $ue^{-\delta \phi} \in L^{2}(\Omega, e^{\psi + (\delta + \delta')\phi}, \omega_{\varepsilon}) \cap A^{2}(\Omega, e^{\psi + (\delta + \delta') \phi}, \omega_{\varepsilon})^{\perp}$.  
It follows that $\overline{\partial}(ue^{-\delta \phi}) = (\alpha-\delta u \overline{\partial}\phi)e^{-\delta \phi} \in L^{2}_{0, 1}(\Omega, e^{\psi + (\delta + \delta')\phi}, \omega_{\varepsilon})$ 
since $|\overline{\partial}\phi|_{\omega_{\varepsilon}} < 1$ and 
$u \in L^{2}(\Omega, e^{\psi + \delta'\phi}, \omega_{\varepsilon})$.   
Then $ue^{-\delta \phi}$ is the minimal solution in the $L^{2}(\Omega, e^{\psi + (\delta + \delta') \phi}, \omega_{\varepsilon})$ norm to $\overline{\partial}(ue^{-\delta \phi})$.  
We note that $A^{-1}_{\varepsilon, \delta + \delta'}$ is bounded from above in $\Omega$.  
By (\ref{equation:2}), we have that 
\begin{align*}
& \int_{\Omega}|u|^{2}e^{\psi-(\delta-\delta') \phi} dV_{\omega_{\varepsilon}} 
= \int_{\Omega} |ue^{-\delta \phi}|^{2} e^{\psi + (\delta + \delta') \phi} dV_{\omega_{\varepsilon}}\\
\leq & \int_{\Omega} \langle A^{-1}_{\varepsilon, \delta+\delta'}(\alpha - \delta u \overline{\partial}\phi), \alpha - \delta u \overline{\partial}\phi\rangle_{\omega_{\varepsilon}} e^{\psi - (\delta-\delta') \phi}dV_{\omega_{\varepsilon}} \\%
\leq & \left(1 + \frac{1}{t}\right) \int_{\Omega} \langle A^{-1}_{\varepsilon, \delta+\delta'}\alpha, \alpha \rangle_{\omega_{\varepsilon}} e^{\psi -(\delta-\delta') \phi}dV_{\omega_{\varepsilon}} + (1 + t) \delta^{2}\int_{\Omega} \langle A^{-1}_{\varepsilon, \delta+\delta'}
\overline{\partial}\phi, \overline{\partial}\phi \rangle_{\omega_{\varepsilon}} |u|^{2}e^{\psi - (\delta-\delta') \phi}dV_{\omega_{\varepsilon}}\nonumber
\end{align*}  
for every $t > 0$.  
Since $\varepsilon < \delta^{-1}$, we have that $\delta \omega_{\varepsilon} = \delta i  \partial\overline{\partial}(\varepsilon\psi + \phi) \leq i \partial \overline{\partial} (\psi + (\delta + \delta') \phi)$.  
By (\ref{equation:1}), 
it follows that $\langle A_{\varepsilon, \delta+\delta'} \beta, \beta \rangle_{\omega_{\varepsilon}} \geq (n-1) \delta |\beta|^{2}_{\omega_{\varepsilon}}$ for any $(0, 1)$-form $\beta$.  
Hence $\langle A^{-1}_{\varepsilon, \delta+\delta'} \overline{\partial} \phi, \overline{\partial}\phi \rangle_{\omega_{\varepsilon}} \leq \frac{1}{(n-1)\delta} |\overline{\partial}\phi|^{2}_{\omega_{\epsilon}} < \frac{1}{(n-1)\delta}$.  
By choosing $t$ so small, there exists a constant $C_{1}$ which depends only on $n$ and $\delta$ such that $(1 + t) \frac{\delta}{n-1} < C_{1}< 1$ since $n> 1 + \delta$.  
Then we have 
\begin{equation*}
(1-C_{1})\int_{\Omega} |u|^{2} e^{\psi - (\delta-\delta') \phi} dV_{\omega_{\varepsilon}} \leq  C_{2} \int_{\Omega} \langle A^{-1}_{\varepsilon, \delta+\delta'}\alpha, \alpha \rangle_{\omega_{\varepsilon}} e^{\psi - (\delta-\delta') \phi}dV_{\omega_{\varepsilon}}.  
\end{equation*}
Here $C_{2} = \left( 1 + \frac{1}{t}\right)$ depends only on $n$ and $\delta$.  
This completes the proof.  
\end{proof}
If there exists a sequence of $\overline{\partial}$-closed $(0, 1)$-forms in $L^{2}_{0, 1}(\Omega, e^{\psi +\delta'\phi}, \omega_{\varepsilon})$ which approximates $\alpha$, we can remove the assumption that $\alpha \in L^{2}_{0, 1}(\Omega, e^{\psi+\delta'\phi}, \omega_{\varepsilon})$ from Lemma~\ref{lemma:1}.  
\begin{lemma}\label{lemma:2}
Let $\delta, \delta' >0$.  
and let $\alpha$ be a $\overline{\partial}$-closed $(0, 1)$-form such that \\$\int_{\Omega}\langle A_{\varepsilon, \delta+\delta'}^{-1}\alpha, \alpha \rangle_{\omega_{\varepsilon}} e^{\psi -(\delta-\delta') \phi}dV_{\omega_{\varepsilon}} < + \infty$.  
Assume that there exist $\overline{\partial}$-closed $(0, 1)$-forms $\alpha_{j} \in L^{2}(\Omega, e^{\psi+\delta'\phi}, \omega_{\varepsilon})$ ($j = 1, 2, \ldots$) such that 
\[
\lim_{j \to \infty} \int_{\Omega} \langle A_{\varepsilon, \delta+\delta'}^{-1}(\alpha-\alpha_{j}), \alpha-\alpha_{j} \rangle_{\omega_{\varepsilon}} e^{\psi -(\delta-\delta') \phi}dV_{\omega_{\varepsilon}} = 0.  
\]
Assume that $n > 1 + \delta$ and that $\varepsilon \leq \delta^{-1}$.  
Then there exists $u \in L^{2}(\Omega, e^{\psi - (\delta-\delta') \phi}, \omega_{\varepsilon})$ such that $\overline{\partial} u = \alpha$ and 
\[
\int_{\Omega}|u|^{2}e^{\psi-(\delta-\delta') \phi}dV_{\omega_{\varepsilon}} \leq C_{n, \delta}\int_{\Omega}\langle A_{\varepsilon, \delta+\delta'}^{-1}\alpha, \alpha \rangle_{\omega_{\varepsilon}} e^{\psi -(\delta-\delta') \phi}dV_{\omega_{\varepsilon}}.  
\]
Here $C_{n, \delta} > 0$ depends only on $n$ and $\delta$.  
\end{lemma}
\begin{proof}
By Lemma~\ref{lemma:1}, there exist $u_{j}$ ($j = 1, 2, \ldots$) such that $\overline{\partial} u_{j} =\alpha_{j}$ and 
\begin{align*}
& \int_{\Omega}|u_{j}|^{2}e^{\psi-(\delta-\delta') \phi}dV_{\omega_{\varepsilon}} 
\leq C_{n, \delta}\int_{\Omega}\langle A_{\varepsilon, \delta+\delta'}^{-1}\alpha_{j}, \alpha_{j} \rangle_{\omega_{\varepsilon}} e^{\psi -(\delta-\delta') \phi}dV_{\omega_{\varepsilon}} \\
\leq & 2C_{n, \delta}\left(\int_{\Omega}\langle A_{\varepsilon, \delta+\delta'}^{-1}\alpha, \alpha \rangle_{\omega_{\varepsilon}} e^{\psi -(\delta-\delta') \phi}dV_{\omega_{\varepsilon}} + \int_{\Omega}\langle A_{\varepsilon, \delta+\delta'}^{-1}(\alpha-\alpha_{j}), \alpha-\alpha_{j} \rangle_{\omega_{\varepsilon}} e^{\psi -(\delta-\delta') \phi}dV_{\omega_{\varepsilon}}\right).  
\end{align*}
Therefore we may choose a subsequence of $\{u_{j}\}_{j \in \mathbb{N}}$ converging weakly in $L^{2}(\Omega, e^{\psi - (\delta-\delta') \phi}, \omega_{\varepsilon})$ to $u$.  
Since $\alpha_{j} \to \alpha$ ($j \to \infty$) in the distribution sense, 
we have that $\overline{\partial} u =\alpha$ and 
\[
\int_{\Omega}|u|^{2}e^{\psi-(\delta-\delta') \phi}dV_{\omega_{\varepsilon}} \leq 2C_{n, \delta}\int_{\Omega}\langle A_{\varepsilon, \delta+\delta'}^{-1}\alpha, \alpha \rangle_{\omega_{\varepsilon}} e^{\psi -(\delta-\delta') \phi}dV_{\omega_{\varepsilon}}.  
\]
\end{proof}
Next, we construct a sequence which approximates $\alpha$.  
\begin{lemma}\label{lemma:3}
Let $\delta, \delta' >0$. 
Let $\alpha \in L^{2}_{0, 1}(\Omega, e^{\psi - (\delta-\delta') \phi}, \omega_{\varepsilon})$ such that $\overline{\partial} \alpha =0$.  
Assume that $n > 2 + \delta$ and that $\varepsilon \leq \delta^{-1}$.  
Then there exist $\overline{\partial}$-closed $(0, 1)$-forms $\alpha_{j} \in L^{2}_{0, 1}(\Omega, e^{\psi + \delta'\phi}, \omega_{\varepsilon})$ ($j = 1, 2, \ldots$) such that 
$\int_{\Omega} \langle A_{\varepsilon, \delta+\delta'}^{-1}(\alpha-\alpha_{j}), \alpha-\alpha_{j} \rangle_{\omega_{\varepsilon}} e^{\psi -(\delta-\delta') \phi}dV_{\omega_{\varepsilon}} \to 0$ when $j \to \infty$.  
\end{lemma}
\begin{proof}
First, note that $\int_{\Omega} \langle A_{\varepsilon, \delta+\delta'}^{-1}\alpha, \alpha \rangle_{\omega_{\varepsilon}} e^{\psi -(\delta-\delta') \phi}dV_{\omega_{\varepsilon}} < + \infty$ since $\langle A_{\varepsilon, \delta+\delta'}^{-1}\alpha, \alpha \rangle_{\omega_{\varepsilon}} \leq \frac{1}{(n-1)\delta}|\alpha|^{2}_{\omega_{\varepsilon}}$ by the proof of Lemma~\ref{lemma:1}.  
Let $\chi \in C^{\infty}(\mathbb{R})$ such that $\chi(t) = 1$ for $t<0$, $\chi(t) = 0$ for $t > 2$ and $|\chi'| \leq 1$.  
Let $h_{j} = \chi(\phi - j) \in C^{\infty}_{0}(\Omega)$.   
Let $N_{1}$, resp.\! $N_{2}$, be the kernel space of linear, closed, densely defined operator $\overline{\partial}: L^{2}_{0, 1}(\Omega, e^{\psi + \delta'\phi}, \omega_{\varepsilon}) \to L^{2}_{0, 2}(\Omega, e^{\psi + \delta'\phi}, \omega_{\varepsilon})$, resp.\! $\overline{\partial}: L^{2}_{0, 1}(\Omega, e^{\psi + (\delta + \delta') \phi}, \omega_{\varepsilon}) \to L^{2}_{0, 2}(\Omega, e^{\psi + (\delta + \delta') \phi}, \omega_{\varepsilon})$.  
We have that $\overline{\partial}(h_{j}\alpha) \in L^{2}_{0, 1}(\Omega, e^{\psi + \delta'\phi}, \omega_{\varepsilon})$.  
By a reasoning analogous to that of the proof of Lemma~\ref{lemma:1}, 
there exists $\beta_{j} \in L^{2}_{0, 1}(\Omega, e^{\psi + \delta' \phi}, \omega_{\varepsilon}) \cap N_{1}^{\perp}$ 
such that 
$\overline{\partial} \beta_{j} = \overline{\partial}(h_{j} \alpha)$, 
$\beta_{j}e^{-\delta \phi} \in L^{2}_{0, 1}(\Omega, e^{\psi + (\delta + \delta')\phi},  \omega_{\varepsilon}) \cap N_{2}^{\perp}$ and 
$\overline{\partial}(\beta_{j} e^{-\delta \phi}) \in L^{2}_{0, 2}(\Omega, e^{\psi + (\delta+\delta') \phi}, \omega_{\varepsilon})$.  
By (\ref{equation:2}), we have that 
\begin{align*}
& \int_{\Omega}|\beta_{j}|^{2}_{\omega_{\varepsilon}}e^{\psi-(\delta-\delta') \phi} dV_{\omega_{\varepsilon}} 
= \int_{\Omega}|\beta_{j}e^{-\delta \phi}|_{\omega_{\varepsilon}}^{2}e^{\psi + (\delta+\delta') \phi}dV_{\omega_{\varepsilon}} \\ 
\leq & \int_{\Omega} \langle A^{-1}_{\varepsilon, \delta+\delta'}(\overline{\partial}\beta_{j} - \delta \overline{\partial}\phi \wedge \beta_{j}), \overline{\partial}\beta_{j} - \delta \overline{\partial}\phi \wedge \beta_{j}\rangle_{\omega_{\varepsilon}} e^{\psi - (\delta-\delta') \phi}dV_{\omega_{\varepsilon}} \\%
\leq & \left(1 + \frac{1}{t}\right) \int_{\Omega} \langle A^{-1}_{\varepsilon, \delta+\delta'}\overline{\partial}\beta_{j}, \overline{\partial}\beta_{j} \rangle_{\omega_{\varepsilon}} e^{\psi - (\delta-\delta') \phi}dV_{\omega_{\varepsilon}} \\
& + (1 + t) \delta^{2}\int_{\Omega} \langle A^{-1}_{\varepsilon, \delta+\delta'}
(\overline{\partial}\phi \wedge \beta_{j}), \overline{\partial}\phi \wedge \beta_{j} \rangle_{\omega_{\varepsilon}} e^{\psi - (\delta-\delta') \phi}dV_{\omega_{\varepsilon}} 
\end{align*}  
for every $t > 0$.  
Here we regard $A_{\varepsilon, \delta+\delta'}$ as an endomorphism of the space of $(0, 2)$-forms.   
Since $\varepsilon < \delta^{-1}$, we have that $\delta \omega_{\varepsilon} = \delta i  \partial\overline{\partial}(\varepsilon\psi + \phi) \leq i \partial \overline{\partial} (\psi + (\delta+\delta') \phi)$.  
By (\ref{equation:1}), 
it follows that $\langle A_{\varepsilon, \delta+\delta'} \gamma, \gamma \rangle_{\omega_{\varepsilon}} \geq (n-2) \delta |\gamma|^{2}_{\omega_{\varepsilon}}$ for any $(0, 2)$-form $\gamma$.  
Hence $\langle A^{-1}_{\varepsilon, \delta+\delta'} (\overline{\partial} \phi \wedge \beta_{j}), \overline{\partial}\phi \wedge \beta_{j} \rangle_{\omega_{\varepsilon}} \leq \frac{1}{(n-2)\delta} |\overline{\partial}\phi \wedge \beta_{j}|^{2}_{\omega_{\epsilon}} \leq \frac{1}{(n-2)\delta} |\overline{\partial}\phi|^{2}_{\omega_{\varepsilon}}|\beta_{j}|_{\omega_{\varepsilon}}^{2} < \frac{1}{(n-2)\delta} |\beta_{j}|_{\omega_{\varepsilon}}^{2}$.  
By choosing $t$ so small, there exists a constant $C_{1}$ which depends only on $n$ and $\delta$ such that $(1 + t) \frac{\delta}{n-2} < C_{1}< 1$ since $n > 2 + \delta$.  
Then we have that 
\[
\int_{\Omega} |\beta_{j}|^{2}_{\omega_{\varepsilon}}e^{\psi - (\delta-\delta') \phi} dV_{\omega_{\varepsilon}} \leq C_{2} \int_{\Omega} \langle A^{-1}_{\varepsilon, \delta+\delta'} \overline{\partial}\beta_{j}, \overline{\partial} \beta_{j} \rangle_{\omega_{\varepsilon}} e^{\psi - (\delta-\delta') \phi} dV_{\omega_{\varepsilon}}
\]
where $C_{2} = (1-C_{1})^{-1}\left(1 + \frac{1}{t} \right)$ which depends only on $n$ and $\delta$.  
It follows that 
\begin{align*}
& \int_{\Omega} \langle A^{-1}_{\varepsilon, \delta+\delta'} \overline{\partial}\beta_{j}, \overline{\partial} \beta_{j} \rangle_{\omega_{\varepsilon}} e^{\psi - (\delta-\delta') \phi} dV_{\omega_{\varepsilon}} \leq \frac{1}{(n-2)\delta} \int_{\Omega} |\overline{\partial}\beta_{j}|^{2}_{\omega_{\varepsilon}} e^{\psi - (\delta-\delta') \phi} dV_{\omega_{\varepsilon}} \\
= & \frac{1}{(n-2)\delta} \int_{\Omega} |\overline{\partial}h_{j}\wedge \alpha|^{2}_{\omega_{\varepsilon}} e^{\psi - (\delta-\delta') \phi} dV_{\omega_{\varepsilon}}
\leq \frac{1}{(n-2)\delta} \int_{\{j \leq \phi \leq j+2\}} |\alpha|^{2}_{\omega_{\varepsilon}}e^{\psi - (\delta-\delta') \phi} dV_{\omega_{\varepsilon}}.  
\end{align*} 
Because $\alpha \in L^{2}_{0, 1}(\Omega, e^{\psi - (\delta-\delta') \phi}, \omega_{\varepsilon})$,  Lebesgue's dominated convergence theorem shows that 
the last term of the above inequality tends to $0$ when $j$ tends to $+\infty$.  
By the proof of Lemma~\ref{lemma:1}, we have that $\langle A^{-1}_{\varepsilon, \delta+\delta'} \beta_{j}, \beta_{j}\rangle_{\omega_{\varepsilon}} \leq \frac{1}{\delta (n-1)}|\beta_{j}|^{2}_{\omega_{\varepsilon}}$.  
Finally, we have that 
\[
\lim_{j \to \infty}\int_{\Omega} \langle A^{-1}_{\varepsilon, \delta+\delta'} \beta_{j}, \beta_{j} \rangle_{\omega_{\varepsilon}} e^{\psi - (\delta-\delta') \phi} dV_{\omega_{\varepsilon}} \leq 
\lim_{j \to \infty} \frac{1}{\delta(n-1)}\int_{\Omega} |\beta_{j}|^{2}_{\omega_{\varepsilon}} e^{\psi - (\delta-\delta') \phi}dV_{\omega_{\varepsilon}} = 0.  
\]
Let $\alpha_{j} = h_{j} \alpha - \beta_{j}$.  
Then $\alpha_{j} \in L^{2}_{0, 1}(\Omega, e^{\psi+\delta'\phi}, \omega_{\varepsilon})$, 
$\overline{\partial}\alpha_{j} = 0$, and 
\begin{align*}
& \lim_{j \to \infty} \int_{\Omega} \langle A_{\varepsilon, \delta+\delta'}^{-1}(\alpha-\alpha_{j}), \alpha-\alpha_{j} \rangle_{\omega_{\varepsilon}} e^{\psi -(\delta-\delta') \phi}dV_{\omega_{\varepsilon}} \\
\leq & \lim_{j \to \infty} 2\int_{\Omega} \left((1-h_{j})^{2} \langle A_{\varepsilon, \delta+\delta'}^{-1}\alpha, \alpha \rangle_{\omega_{\varepsilon}} + \langle A^{-1}_{\varepsilon, \delta+\delta'} \beta_{j}, \beta_{j} \rangle_{\omega_{\varepsilon}}\right) e^{\psi - (\delta-\delta') \phi}  dV_{\omega_{\varepsilon}} = 0 
\end{align*}
by Lebesgue's dominated convergence theorem.  
\end{proof}

\begin{lemma}\label{lemma:4}
Let $\delta, \delta' >0$.  
Let $k < n-1$.  
Let $\alpha$ be a smooth $(0, 1)$-form on $\Omega$ such that $\mathrm{supp}\, \alpha \subset \Omega \setminus \mathrm{supp}\, (i\partial \overline{\partial}\varphi)^{k}$.  
Then $\langle A^{-1}_{\varepsilon, \delta+\delta'}\alpha, \alpha \rangle_{\omega_{\varepsilon}} \leq (n-k-1)^{-1} |\alpha|_{i\partial \overline{\partial} \psi}^{2}$.  
\end{lemma}
\begin{proof}
We have that $A^{-1}_{\varepsilon, \delta+\delta'} \leq A^{-1}_{\varepsilon, \delta}$.  
Hence it is enough to prove $\langle A^{-1}_{\varepsilon, \delta}\alpha, \alpha \rangle_{\omega_{\varepsilon}} \leq (n-k-1)^{-1} |\alpha|_{i\partial \overline{\partial} \psi}^{2}$. 
Because $i\partial \overline{\partial} \phi = i\frac{\partial \overline{\partial} \varphi}{-\varphi} + i \frac{\partial \varphi \wedge \overline{\partial}\varphi}{\varphi^{2}}$, 
we have that $\mathrm{supp}\,\alpha \subset \Omega \setminus \mathrm{supp}\,(i\partial \overline{\partial}\phi)^{k+1}$.  
Let $x \in \Omega \setminus \mathrm{supp}\,(i\partial \overline{\partial}\phi)^{k+1}$.  
At $x$, we choose an orthonormal basis $\theta_{1}, \ldots, \theta_{n}$ for the holomorphic cotangent bundle with respect to $i \partial \overline{\partial} \psi$ such that $i \partial \overline{\partial} \phi = i \lambda_{1} \theta_{1}\wedge \overline{\theta}_{1} + \cdots + i\lambda_{k}\theta_{k}\wedge\overline{\theta_{k}}$ where $\lambda_{j} \geq 0$ for $1 \leq j \leq k$.  
Then $\omega_{\varepsilon} = i \partial \overline{\partial}(\varepsilon \psi + \phi) = \sum_{j= 1}^{k}i(\varepsilon + \lambda_{j})\theta_{j} \wedge \overline{\theta_{j}} + \sum_{l=k+1}^{n} i\varepsilon \theta_{l} \wedge \overline{\theta_{l}}$ and 
$i\partial \overline{\partial} (\psi + \delta \phi) = \sum_{j = 1}^{k} i (1 + \delta \lambda_{j})\theta_{j} \wedge \overline{\theta_{j}} + \sum_{l = k+1}^{n} i \theta_{l} \wedge \overline{\theta_{l}}$.  
Let $\sigma_{j} = \sqrt{\varepsilon + \lambda_{j}} \theta_{j}$ for $1 \leq j \leq k$ and let $\sigma_{l} = \sqrt{\varepsilon} \theta_{l}$ for $k +1 \leq  l \leq n$.  
Then $\omega_{\varepsilon} = \sum_{j = 1}^{n} i \sigma_{j} \wedge \overline{\sigma_{j}}$ and $i\partial\overline{\partial}(\psi + \delta \phi) = \sum_{j = 1}^{k} i \frac{1+\delta \lambda_{j}}{\varepsilon + \lambda_{j}}\sigma_{j} \wedge \overline{\sigma_{j}} + \sum_{l = k+1}^{n} i \frac{1}{\varepsilon}\sigma_{l} \wedge \overline{\sigma_{l}}$.  
By (\ref{equation:1}), it follows that 
\[
\langle A^{-1}_{\varepsilon, \delta} \overline{\sigma_{j}}, \overline{\sigma_{s}}\rangle_{\omega_{\varepsilon}} = \langle \biggl( \sum_{\substack{1 \leq m \leq k \\ m \neq j}} \frac{1 + \delta \lambda_{m}}{\varepsilon + \lambda_{m}} + \frac{n-k}{\varepsilon} \biggl)^{-1}\overline{\sigma_{j}}, \overline{\sigma_{s}}\rangle_{\omega_{\varepsilon}} \leq \frac{\varepsilon}{n-k}\delta_{j s} 
\]
for $j \leq k$, $1 \leq s \leq n$, and 
\[
\langle A^{-1}_{\varepsilon, \delta} \overline{\sigma_{l}}, \overline{\sigma_s}\rangle_{\omega_{\varepsilon}}= \langle \biggl( \sum_{1 \leq m \leq k} \frac{1 + \delta \lambda_{m}}{\varepsilon + \lambda_{m}} + \frac{n-k-1}{\varepsilon} \biggl)^{-1}\overline{\sigma_{l}}, \overline{\sigma_{s}}\rangle_{\omega_{\varepsilon}} \leq \frac{\varepsilon}{n-k-1} \delta_{l s} 
\]
for $l \geq k+1$, $1 \leq s \leq n$.  
Here $\delta_{j s}$, $\delta_{l s}$ are the Kronecker delta.  
Write $\alpha = \sum_{j = 1}^{n} \alpha_{j} \overline{\theta_{j}}= \sum_{j = 1}^{k} \frac{\alpha_{j}}{\sqrt{\varepsilon + \lambda_{j}}} \overline{\sigma_{j}} + \sum_{l = k+1}^{n} \frac{\alpha_{l}}{\sqrt{\varepsilon}} \overline{\sigma_{l}}$.  
We have that 
\[
\langle A^{-1}_{\varepsilon, \delta}\alpha, \alpha \rangle_{\omega_{\varepsilon}} 
\leq \sum_{ j = 1}^{k}\frac{|\alpha_{j}|^{2}}{\varepsilon + \lambda_{j}}\frac{\varepsilon}{n-k} + \sum_{l = k +1}^{n} \frac{|\alpha_{l}|^{2}}{\varepsilon}\frac{\varepsilon}{n-k-1} \leq 
(n-k-1)^{-1} |\alpha|_{i\partial \overline{\partial} \psi}^{2}. 
\]
\end{proof}

\begin{lemma}\label{lemma:5}
Assume that $\varphi \in C^{\infty}(\overline{\Omega})$ and that $d\varphi \neq 0$ on $\partial \Omega$.  
Let $p \in \partial\Omega$ and let $1 \leq k \leq n$ be an integer.  
Assume that $(i\partial \overline{\partial} \varphi)^{k} =0$ in a neighborhood of $p$.  
If $\delta > k$, then $e^{\psi - \delta \phi}dV_{\omega_{\varepsilon}}$ is integrable around $p$ in $\Omega$.  
\end{lemma}
\begin{proof}
Let $C_{1}, C_{2},\ldots$ be sufficiently large positive constants.  
Since $\omega_{\varepsilon} = i \varepsilon \partial \overline{\partial} \psi + i \frac{\partial \overline{\partial}\varphi}{-\varphi} + i \frac{\partial \varphi \wedge \overline{\partial}\varphi}{\varphi^{2}}$, 
we have that 
$dV_{\omega_{\varepsilon}} = (n!)^{-1}\omega_{\varepsilon}^{n} \leq C_{1} (-\varphi)^{-k-1} (i\partial \overline{\partial}|z|^{2})^{n}$ 
and $e^{\psi - \delta \phi} dV_{\omega_{\varepsilon}} \leq C_{2} (- \varphi)^{\delta - k -1}(i \partial \overline{\partial}|z|^{2})$ around $p$.   
Let $U$ be a small neighborhood of $p$.  
There exists a local coordinate system $(x_{1}, \ldots, x_{2n})$ on $U$ such that $\varphi = x_{1}$.  
It follows that 
\[
\int_{\Omega \cap U} e^{\psi - \delta \phi}dV_{\omega_{\varepsilon}} \leq C_{3} \int_{\Omega \cap U} x_{1}^{\delta - k -1}dx_{1}\cdots dx_{2n} \leq C_{4}\int_{0}^{1} x_{1}^{\delta - k - 1} dx_{1} < + \infty 
\]
since $\delta > k$.  
\end{proof}

\begin{proposition}\label{proposition:1}
Let $\Omega \subset \mathbb{C}^{n}$ ($n \geq 4$) be a bounded hyperconvex domain 
and let $\psi$ be a smooth strongly plurisubharmonic function on a neighborhood of $\overline{\Omega}$.  
Let $\varphi \in C^{\infty}(\overline{\Omega})$ such that 
$\varphi$ is negative plurisubharmonic on $\Omega$, $\varphi(z) \to 0$ when $z \to \partial \Omega$, and $d\varphi \neq 0$ on $\partial \Omega$.  
Let $\alpha$ be a smooth $(0, 1)$-form defined on an open neighborhood of $\overline{\Omega}$ such that $\overline{\partial}\alpha = 0$ in $\Omega$ and 
$\mathrm{supp}\, \alpha \subset \overline{\Omega} \setminus \mathrm{supp}\, (i \partial \overline{\partial}\varphi)^{n-3}$ in $\overline{\Omega}$.   
Let $\delta$ be a positive constant such that $n-3 < \delta < n-2$.  
Then there exists $u \in C^{\infty}(\Omega)$ such that $\overline{\partial}u = \alpha$ and 
\[
\int_{\Omega}|u|^{2} e^{\psi - \delta \phi} dV_{\omega_{\varepsilon}} \leq C_{n, \delta} \int_{\Omega} |\alpha|^{2}_{i\partial \overline{\partial}\psi} e^{\psi - \delta \phi}dV_{\omega_{\varepsilon}} < +\infty  
\]
for sufficiently small $\varepsilon>0$. 
Here $C_{n, \delta}$ is a positive constant which depends only on $n$ and $\delta$.   
\end{proposition}
\begin{proof}
Since $|\alpha|^{2}_{\omega_{\varepsilon}} \leq |\alpha|^{2}_{\varepsilon i\partial\overline{\partial}\psi}$, the norm $|\alpha|^{2}_{\omega_{\varepsilon}}$ is bounded from above in $\Omega$.  
Then $|\alpha|^{2}_{\omega_{\varepsilon}}e^{\psi - \delta \phi}dV_{\omega_{\varepsilon}}$ is integrable by Lemma~\ref{lemma:5}.  
Let $\delta'>0$ be a sufficiently small positive number such that $\delta + \delta'<n-2$.  
We put $\delta'' = \delta + \delta'$.  
Then $\delta', \delta''$ depend only on $n$ and $\delta$.  
We have 
$\alpha \in L^{2}(\Omega, e^{\psi - (\delta'' - \delta')\phi}, \omega_{\varepsilon})$.  
By replacing $\delta$ with $\delta''$ in Lemma~\ref{lemma:2}, \ref{lemma:3} and \ref{lemma:4}, 
it follows that there exists $u \in L^{2}(\Omega, e^{\psi - (\delta''-\delta')\phi}, \omega_{\varepsilon})$ such that $\overline{\partial} u = \alpha$ and 
\[
\int_{\Omega} |u|^{2}e^{\psi - (\delta''-\delta')\phi}dV_{\omega_{\varepsilon}} \leq C_{n, \delta} \int_{\Omega}|\alpha|^{2}_{i\partial \overline{\partial}\psi}e^{\psi - (\delta''-\delta')\phi}dV_{\omega_{\varepsilon}}.  
\]
Then we have the proposition since $\delta = \delta''-\delta'$.    
The smoothness of $u$ is known (see \cite{Dem}, \cite{Hor}).  
\end{proof}

\section{Interior estimate of non-negative plurisubharmonic functions}\label{section:4}
The purpose of this section is the following theorem:
\begin{theorem}\label{theorem:2}
Let $\Omega \subset \mathbb{C}^{n}$ be a bounded hyperconvex domain and let $\varphi$ be a negative continuous plurisubharmonic function on $\Omega$ such that $\varphi (z) \to 0$ when $z \to \partial \Omega$.  
Let $v \geq 0$ be a plurisubharmonic function on $\Omega$.  
Then 
\begin{align*} 
&\int_{\{\varphi < r\}} i \partial \overline{\partial}v \wedge (i \partial \overline{\partial}|z|^{2})^{n-1} 
\leq C \int_{\Omega} v i \partial \varphi \wedge \overline{\partial} \varphi \wedge (i \partial \overline{\partial} \varphi )^{n-1}, \\
&\int_{\{\varphi < r\}} v (i \partial \overline{\partial}|z|^{2})^{n} 
\leq C \int_{\Omega} \left( v (i \partial \overline{\partial}\varphi)^{n} + v i \partial \varphi \wedge \overline{\partial} \varphi \wedge (i \partial \overline{\partial} \varphi )^{n-1}\right).   
\end{align*}
for $r < 0$.  
Here $C = \left( 1 + d(\Omega) + \sup |\varphi| + |r|^{-1} \right)^{C_{n}}$, 
$d(\Omega)$ is a diameter of $\Omega$, and $C_{n}$ is a positive constant which depends only on $n$.  
\end{theorem}
In the above theorem, $i\partial \overline{\partial}v \wedge (i\partial \overline{\partial}|z|^{2})^{n-1}$, $i \partial \varphi \wedge \overline{\partial} \varphi \wedge (i \partial \overline{\partial} \varphi )^{n-1}$, and $(i \partial \overline{\partial}\varphi)^{n}$ are defined in the sense of Bedford-Taylor (see \cite{Bed-Tay}, \cite{Kli}).  
 
\begin{lemma}\label{lemma:6}
Let $k$ be a non-negative integer.  
We assume the same hypothesis of Theorem~\ref{theorem:2}, and we 
assume that $v, \varphi \in C^{\infty}(\overline{\Omega})$ and that $d\varphi \neq 0$ on $\partial \Omega$.  
Then 
\[
\int_{\{\varphi < r\}}i \partial \overline{\partial}v \wedge (i \partial \overline{\partial}\varphi)^{k}\wedge(i\partial \overline{\partial}|z|^{2})^{n-k-1} 
\leq C_{n, k} \frac{(d(\Omega)^{2}\sup |\varphi|)^{n-k-1}}{r^{2(n-k)}}\int_{\Omega} v i\partial \varphi \wedge \overline{\partial}\varphi \wedge (i\partial \overline{\partial}\varphi)^{n-1}.   
\]
Here $C_{n, k}$ is a positive constant which depends only on $n$ and $k$.  
\end{lemma}
\begin{proof}
By Stokes theorem, we have that 
\begin{align*}
&\int_{\Omega} v i \partial \varphi \wedge \overline{\partial}\varphi \wedge (i \partial \overline{\partial}\varphi)^{n-1} = - \int_{\Omega} \varphi i\partial v \wedge \overline{\partial}\varphi \wedge (i\partial \overline{\partial}\varphi)^{n-1} + \int_{\Omega} (-\varphi) v (i\partial \overline{\partial}\varphi)^{n}\\
\geq & -\frac{1}{2} \int_{\Omega} i\partial v \wedge \overline{\partial}\varphi^{2} \wedge (i\partial \overline{\partial}\varphi)^{n-1} = \frac{1}{2} \int_{\Omega} \varphi^{2} i \partial \overline{\partial}v \wedge (i \partial \overline{\partial}\varphi)^{n-1}, 
\end{align*}
and we have that 
\begin{equation}\label{equation:3}
\int_{\Omega} \varphi^{2} i \partial \overline{\partial}v \wedge (i \partial \overline{\partial}\varphi)^{n-1} \leq 2 \int_{\Omega} v i \partial \varphi \wedge \overline{\partial} \varphi \wedge (i \partial \overline{\partial}\varphi)^{n-1}.  
\end{equation}

Without loss of generality, we may assume that $0 \in \partial\Omega$.  
Let $\eta = \frac{|r|}{2d(\Omega)^{2}}(|z|^{2}-2d(\Omega)^{2})$.  
We have that $\eta$ is smooth plurisubharmonic function such that $r < \eta < \frac{r}{2}$ in $\Omega$. 
For sufficiently small $\epsilon > 0$, we put $\rho = \max_{\epsilon}\{\varphi, \eta\}$ where $\max_{\epsilon}$ is a regularized max function (see Chapter~I, Section~5 of \cite{Dem}).  
Then $\rho$ is a smooth plurisubharmonic function on $\Omega$ such that 
$\rho = \varphi$ near $\{z \in \Omega \, |\, \varphi(z) = \frac{r}{3}\}$ and 
$\rho = \eta$ on $\{z \in \Omega \, |\, \varphi(z) < r\}$.  
After a slight perturbation of $r$, we may assume that $d\varphi \neq 0$ on $\{z \in \Omega \, |\,  \varphi(z) = \frac{r}{3}\}$.  
By Stokes theorem, we have that 
\begin{align*}
& \int_{\{\varphi < r\}} i \partial \overline{\partial}v \wedge (i \partial \overline{\partial}\varphi)^{k} \wedge (i\partial \overline{\partial}|z|^{2})^{n-k-1}  \\
= & \frac{2d(\Omega)^{2}}{|r|} \int_{\{\varphi < r\}} i\partial \overline{\partial} v \wedge  (i \partial \overline{\partial}\varphi)^{k} \wedge i \partial \overline{\partial}\rho\wedge (i\partial \overline{\partial}|z|^{2})^{n-k-2} \\
\leq & \frac{2d(\Omega)^{2}}{|r|}\frac{3}{2|r|}\int_{\{\varphi < r\}} \left( \frac{r}{3}-\varphi \right) i \partial \overline{\partial} v \wedge (i \partial \overline{\partial}\varphi)^{k} \wedge i\partial \overline{\partial}\rho \wedge (i\partial \overline{\partial}|z|^{2})^{n-k-2} \\
\leq & \frac{3d(\Omega)^{2}}{|r|^{2}} \int_{\{\varphi < r/3 \}} \left( \frac{r}{3}-\varphi \right) i \partial \overline{\partial} v \wedge (i \partial \overline{\partial}\varphi)^{k} \wedge i\partial \overline{\partial}\left( \rho - \frac{r}{3}\right) \wedge (i\partial \overline{\partial}|z|^{2})^{n-k-2} \\
= & \frac{3d(\Omega)^{2}}{|r|^{2}} \int_{\{\varphi < r/3 \}} \left( \frac{r}{3}-\rho \right) i \partial \overline{\partial} v \wedge (i \partial \overline{\partial}\varphi)^{k} \wedge i\partial \overline{\partial}\left( \varphi - \frac{r}{3}\right) \wedge (i\partial \overline{\partial}|z|^{2})^{n-k-2} \\
\leq & \frac{3d(\Omega)^{2}\sup |\varphi|}{|r|^{2}} \int_{\{\varphi < r/3 \}} i \partial \overline{\partial} v \wedge (i \partial \overline{\partial}\varphi)^{k+1} \wedge (i\partial \overline{\partial}|z|^{2})^{n-k-2} 
\end{align*}
By repeating the same process, we have that 
\begin{align*}
& \int_{\{\varphi < r\}} i \partial \overline{\partial}v \wedge (i \partial \overline{\partial}\varphi)^{k} \wedge (i\partial \overline{\partial}|z|^{2})^{n-k-1}  \\
\leq & 3^{(n-k-1)^{2}}\left(\frac{d(\Omega)^{2} \sup |\varphi|}{|r|^{2}}\right)^{n-k-1} 
\int_{\{\varphi < r/3^{n-k-1}\}} i\partial \overline{\partial}v \wedge (i \partial \overline{\partial} \varphi)^{n-1} \\
\leq & 3^{(n-k-1)^{2}}\left(\frac{d(\Omega)^{2} \sup |\varphi|}{|r|^{2}}\right)^{n-k-1} 
\left( \frac{3^{n-k-1}}{|r|}\right)^{2} \int_{\{\varphi < r/3^{n-k-1}\}} \varphi^{2} i\partial \overline{\partial}v \wedge (i \partial \overline{\partial} \varphi)^{n-1} \\
\leq & 3^{(n-k-1)(n-k +1)}2\frac{(d(\Omega)^{2} \sup |\varphi|)^{n-k-1}}{|r|^{2(n-k)}}
\int_{\Omega} v i \partial \varphi \wedge \overline{\partial} \varphi \wedge (i \partial \overline{\partial}\varphi)^{n-1}.  
\end{align*}
The last inequality follows from (\ref{equation:3}).  
This completes the proof.  
\end{proof}
\begin{remark}
To prove Theorem~\ref{theorem:1}, the rest of this section is not necessary.  
Indeed, Lemma~\ref{lemma:6} shows that 
\[
\int_{\{\varphi<r\}} |\nabla F|^{2} (i\partial \overline{\partial} |z|^{2})^{n} \leq C \int_{\Omega} |F|^{2}i\partial \varphi \wedge \overline{\partial} \varphi \wedge (i \partial \overline{\partial} \varphi)^{n-1}
\]
for holomorphic function $F$.  
Here $C$ does not depend on $F$.  
This implies that the solutions constructed in Section~\ref{section:3} are bounded locally and we can prove Theorem~\ref{theorem:3} below.  
\end{remark}
\begin{lemma}\label{lemma:7}
Let $k$ be a non-negative integer.  
Under the same assumption of Lemma~\ref{lemma:6}, we have that 
\[
\int_{\{\varphi < r\}} v (i \partial \overline{\partial}\varphi)^{k} \wedge (i \partial \overline{\partial} |z|^{2})^{n-k} 
\leq C \left( \int_{\Omega} v(i\partial \overline{\partial} \varphi)^{n} + v i\partial \varphi \wedge \overline{\partial} \varphi \wedge (i \partial \overline{\partial} \varphi)^{n-1}\right)
\]
where $C = \left(1 + d(\Omega) + \sup |\varphi| + |r|^{-1} \right)^{C_{n, k}}$, and $C_{n, k}$ is a positive constant which depends only on $n$ and $k$.  
\end{lemma}
\begin{proof}
We prove the lemma by induction on $l = n-k$.  
It is clear for $l = 0$.  
Under the notation of the proof of Lemma~\ref{lemma:6}, we have that 
\begin{align*}
& \frac{|r|}{2d(\Omega)^{2}}\int_{\{\varphi < r\}} v (i \partial \overline{\partial}\varphi)^{k} \wedge (\partial \overline{\partial}|z|^{2})^{n-k} 
\leq   \int_{\{\varphi < r/3\}} v (i \partial \overline{\partial}\varphi)^{k} \wedge i\partial \overline{\partial}\rho \wedge (i\partial \overline{\partial}|z|^{2})^{n-k-1} \\ \nonumber
= & \int_{\{\varphi = r/3\}} v i \overline{\partial}\rho \wedge (i \partial \overline{\partial}\varphi)^{k} \wedge (i \partial \overline{\partial} |z|^{2})^{n-k-1}
-\int_{\{\varphi < r/3\}} i \partial v \wedge \overline{\partial}\rho \wedge (i\partial \overline{\partial}\varphi)^{k} \wedge (i\partial \overline{\partial}|z|^{2})^{n-k-1}\\
= &  \int_{\{\varphi = r/3\}} v i \overline{\partial}\varphi \wedge (i \partial \overline{\partial}\varphi)^{k} \wedge (i \partial \overline{\partial} |z|^{2})^{n-k-1} 
-\int_{\{\varphi < r/3\}} i \partial v \wedge \overline{\partial}\rho \wedge (i\partial \overline{\partial}\varphi)^{k} \wedge (i\partial \overline{\partial}|z|^{2})^{n-k-1}\\
= & \int_{\{\varphi < r/3\}} i \partial v \wedge \overline{\partial}(\varphi - \rho) \wedge (i\partial\overline{\partial}\varphi)^{k} \wedge (i\partial \overline{\partial}|z|^{2})^{n-k-1} 
+ \int_{\{\varphi < r/3\}} v (i\partial \overline{\partial} \varphi)^{k+1} \wedge (i \partial \overline{\partial}|z|^{2})^{n-k-1}.  
\end{align*}
The last term of the above inequality is bounded from above by the hypothesis of the induction.  
By Lemma~\ref{lemma:6}, 
the second to last term of the above inequality is bounded from above by 
\[
C_{n, k} \frac{d(\Omega)^{2(n-k-1)}\sup |\varphi|^{n-k}}{r^{2(n-k)}}\int_{\Omega} v i\partial \varphi \wedge \overline{\partial}\varphi \wedge (i\partial \overline{\partial}\varphi)^{n-1}
\]
since 
\begin{align*}
& \left| \int_{\{\varphi < r/3\}} (\rho - \varphi) i\partial \overline{\partial}v \wedge (i \partial \overline{\partial} \varphi)^{k}\wedge (i\partial \overline{\partial}|z|^{2})^{n-k-1} \right|\\
\leq & \sup |\varphi| \int_{\{\varphi < r/3\}} i\partial \overline{\partial}v \wedge (i \partial \overline{\partial} \varphi)^{k}\wedge (i\partial \overline{\partial}|z|^{2})^{n-k-1}.  
\end{align*}
This completes the proof by the induction.  
\end{proof}
\begin{proof}[Proof of Theorem \ref{theorem:2}]
We prove the first inequality.  
Let $\varepsilon > 0$ be a small positive number.  
It is enough to prove the theorem with $\varphi$ and $\Omega$ replaced by 
$\varphi + \varepsilon$ and $\{z \in \Omega \,|\, \varphi(z) + \varepsilon < 0\}$.  
Hence we may assume that $\varphi$ and $v$ are plurisubharmonic functions defined on an open neighborhood of $\overline{\Omega}$.  
Let $v_{j}$ be a decreasing sequence of smooth plurisubharmonic functions on an open  neighborhood of $\overline{\Omega}$ which converge to $v$.  
Since $\int_{\{\varphi < r\}} i\partial \overline{\partial}v \wedge (i\partial \overline{\partial}|z|^{2})^{n-1} \leq \liminf_{j \to \infty} \int_{\{\varphi < r\}}i\partial \overline{\partial}v_{j} \wedge (i\partial \overline{\partial}|z|^{2})^{n-1}$, 
it is enough to prove the theorem for $v \in C^{\infty}(\overline{\Omega})$.  
Since $\varphi$ is continuous, there exists a decreasing sequence $\varphi_{j}$ of smooth plurisubharmonic functions on an open  neighborhood of $\overline{\Omega}$ which converge to $\varphi$ uniformly.  
Let $\Omega_{j} = \{z \in \Omega \, | \, \varphi_{j}(z) < 0\}$.  
We may assume that $d\varphi_{j} \neq 0$ on $\partial \Omega_{j}$ by Sard's theorem.  
By Lemma~\ref{lemma:6}, 
we have 
\[
\int_{\{\varphi_{j} < r\}} i \partial \overline{\partial}v \wedge (i \partial \overline{\partial}|z|^{2})^{n-1} 
\leq C_{n} \frac{(d(\Omega_{j})^{2} \sup |\varphi_{j}|)^{n-1}}{r^{2n}} \int_{\Omega_{j}} v i \partial \varphi_{j} \wedge \overline{\partial} \varphi_{j} \wedge (i \partial \overline{\partial} \varphi_{j})^{n-1} 
\]
Since 
\[
\limsup_{j \to \infty} \int_{\overline{\Omega}} v i \partial \varphi_{j} \wedge \overline{\partial} \varphi_{j} \wedge (i \partial \overline{\partial} \varphi_{j} )^{n-1} 
\leq \int_{\overline{\Omega}} v i \partial \varphi \wedge \overline{\partial} \varphi \wedge (i \partial \overline{\partial} \varphi )^{n-1}, 
\]
we have that 
\[
\int_{\{\varphi < r\}} i \partial \overline{\partial}v \wedge (i \partial \overline{\partial}|z|^{2})^{n-1} 
\leq C_{n} \frac{(d(\Omega)^{2} \sup |\varphi|)^{n-1}}{r^{2n}} \int_{\overline{\Omega}} v i \partial \varphi \wedge \overline{\partial} \varphi \wedge (i \partial \overline{\partial} \varphi)^{n-1}.  
\] 
Then the first inequality of Theorem~\ref{theorem:2} follows 
by the continuity of  $C$ in the theorem with respect to $d(\Omega), \sup |\varphi|$, and $|r|$.  
The second inequality can be proved by the same way.  
\end{proof}
\section{Proof of the main theorem}
Now we solve the $\overline{\partial}$ equation in the $L^{2}$-space defined by $(i\partial \overline{\partial} \phi)^{n}$.  
\begin{theorem}\label{theorem:3}
Let $\Omega \subset \mathbb{C}^{n}$ ($n \geq 4$) be a bounded hyperconvex domain.  
Let $\varphi \in C^{\infty}(\overline{\Omega})$ such that 
$\varphi$ is negative plurisubharmonic on $\Omega$, $\varphi(z) \to 0$ when $z \to \partial \Omega$, and $d\varphi \neq 0$ on $\partial \Omega$.  
Let $\alpha$ be a smooth $(0, 1)$-form defined on an open neighborhood of $\overline{\Omega}$ such that $\overline{\partial}\alpha = 0$ in $\Omega$ and 
$\mathrm{supp}\, \alpha \subset \overline{\Omega} \setminus \mathrm{supp}\, (i \partial \overline{\partial}\varphi)^{n-3}$ in $\overline{\Omega}$.   
Then there exists a smooth function $u$ on $\Omega$ such that $\overline{\partial} u = \alpha$ and $\int_{\Omega} |u|^{2} (i\partial \overline{\partial} \phi)^{n} = 0$.  
\end{theorem}
\begin{proof}
We use the same notation as Proposition~\ref{proposition:1}.  
Let $\varepsilon_{j}$ be a decreasing sequence of positive numbers which converge to $0$.  
We put $dV_{j} = (n!)^{-1}\omega_{\varepsilon_{j}}^{n}$.  
Then $dV_{j}$ decreases to $dV_{i \partial \overline{\partial}\phi} = (n!)^{-1}(i\partial \overline{\partial}\phi)^{n}$.  
By Proposition~\ref{proposition:1}, 
there exists a sequence $u_{j}$ of smooth functions such that $\overline{\partial}u_{j} = \alpha$ and 
\[
\int_{\Omega} |u_{j}|^{2} e^{\psi - \delta \phi}dV_{j} \leq C_{n, \delta} \int_{\Omega}|\alpha|^{2}_{i\partial \overline{\partial}\psi} e^{\psi - \delta \phi}dV_{j}.  
\]
We have that $\mathrm{supp}\, \alpha \subset \Omega \setminus \mathrm{supp}\, (i\partial \overline{\partial}\varphi)^{n-3} \subset \Omega \setminus \mathrm{supp}\, (i\partial \overline{\partial}\phi)^{n}$.  
Hence the right hand side of the above inequality goes to $0$ when $j \to \infty$ because of Lebesgue's dominated convergence theorem.  
Let $\Omega(r) = \{z \in \Omega \, |\, \varphi(z) < r\}$.  
Then $\int_{\Omega(r)}|u_{j}|^{2} dV_{i\partial \overline{\partial}\phi}$ goes to $0$ when $j \to \infty$ for $r<0$.  
We take $h \in C^{\infty}(\Omega)$ such that $\overline{\partial} h = \alpha$ (see \cite{Dem}, \cite{Hor}).  
Define $F_{j} = h - u_{j}$.  
Then $F_{j}$ is a holomorphic function and $\int_{\Omega(r)} |F_{j}|^{2} dV_{i \partial \overline{\partial}\phi}$ are bounded from above for all $j$.  
Since there exists a positive constant $C$ such that $i\partial \varphi \wedge \overline{\partial}\varphi \wedge (i \partial \overline{\partial} \varphi)^{n-1} + (i\partial \overline{\partial}\varphi)^{n} \leq C(i\partial \overline{\partial}\phi)^{n}$ on $\Omega(r)$, 
Theorem~\ref{theorem:2} shows that $\int_{\Omega(r')} |F_{j}|^{2} (i \partial \overline{\partial} |z|^{2})^{n}$ and $\int_{\Omega(r')} |u_{j}|^{2} (i \partial \overline{\partial} |z|^{2})^{n}$ are bounded from above for all $j$ when $r' < r$.  
We can thus find a weakly convergent subsequence $u_{j_{\nu}}$ in $L^{2}(\Omega(r'))$.  
Let $u$ be the weak limit $u$.  
It follows that $\overline{\partial}u = \alpha$ on $\Omega(r')$ and $\int_{\Omega(r')} |u|^{2} (i\partial \overline{\partial} \phi)^{n} = 0$.  
Then, by using a diagonal argument, we have the solution we are looking for.  
\end{proof}

\begin{proof}[Proof of Theorem~\ref{theorem:1}]
Let $r <0$ such that $|r|$ is sufficiently small and let $\Omega(r) = \{z \in \Omega \, |\, \varphi(z) < r\}$.  
We can choose $r$ such that $d \varphi \neq 0$ on $\partial \Omega(r)$.  
Let $V(r) = \Omega(r) \cap V$.  
There exists $\delta >0$ such that 
$d(\partial V(r) \setminus \partial \Omega(r), \mathrm{supp}\, (i \partial \overline{\partial}\varphi)^{n-3} \cap \Omega(r)) > 3\delta$.  
Here $d(A, B)$, $A, B \subset \mathbb{C}^{n}$ is the Euclidean distance between $A$ and $B$.  
Let $U_{j} = \{z \in \Omega(r)\, |\, d(z, \mathrm{supp}\, (i\partial \overline{\partial}\varphi)^{n-3} \cap \Omega(r)) < j \delta\}$ ($j = 1, 2$).  
We take a smooth function $\chi$ on $\Omega(r)$ such that $\chi = 1$ on $U_{1}$ and $\chi = 0$ on $\Omega(r) \setminus U_{2}$.  
Let $f$ be a holomorphic function on $V$.  
Define $\alpha = \overline{\partial} (\chi f)$.  
We may assume that $\alpha$ is defined on an open neighborhood of $\overline{\Omega(r)}$ by a small perturbation of $r$.  
Then $\mathrm{supp}\, \alpha \subset \overline{\Omega(r)} \setminus \mathrm{supp}\, (i \partial \overline{\partial}\varphi)^{n-3}$ in $\overline{\Omega(r)}$.  
By Theorem~\ref{theorem:3}, there exists $u \in C^{\infty}(\Omega(r))$ such that $\overline{\partial} u = \alpha$ and 
$\int_{\Omega(r)}|u|^{2} (i \partial \overline{\partial}(-(\log(r- \varphi)))^{n} = 0$.  
(If $\Omega(r)$ is a disjoint union of bounded hyperconvex domain, we apply Theorem~\ref{theorem:3} to each component.)  
Then $u = 0$ on $\mathrm{supp}\, (i \partial \overline{\partial} \phi)^{n} \cap \Omega(r)$ since $\mathrm{supp}\, (i \partial \overline{\partial} \phi)^{n} = \mathrm{supp}\, (i\partial \overline{\partial}(-\log(r - \varphi)))^{n}$.   
Let $F_{r} = \chi f - u$.  
Then $F_{r}$ is holomorphic on $\Omega(r)$ and $F_{r} = f$ on $\mathrm{supp}\, (i \partial \overline{\partial}\phi)^{n}\cap \Omega(r)$.  
We note that any component of $\Omega(r)$ intersects $\mathrm{supp}\, (i \partial \overline{\partial}\phi)^{n}$ by the comparison theorem (see \cite{Kli}).  
By letting $r \to 0$, we obtain the holomorphic function $F$ on $\Omega$ such that $F = f$ on $\mathrm{supp}\, (i \partial \overline{\partial}\phi)^{n}$ because of the identity theorem.  
Since $\mathrm{supp}\, (i\partial \overline{\partial}\phi)^{n} \subset V$ and $V$ is connected, we have $F = f$ on $V$.  
\end{proof}
\begin{proof}[Proof of Corollary~\ref{corollary:1}]
We use the same notation as the proof of Theorem~\ref{theorem:1}.  
Let $p \in \mathrm{supp}\, (i\partial \overline{\partial}\phi)^{n} \subset V$.  
Let $h:\mathbb{C}^{n} \to \mathbb{R}^{+}$ be a smooth function of $|z|$ whose support is the unit ball and whose integral is equal to one.  
Define $h_{\varepsilon} = (1/\varepsilon^{2n})h(z/\varepsilon)$ for $\varepsilon > 0$.  
Let $\varphi_{\varepsilon} = \varphi * h_{\varepsilon}$ be a function on $\Omega(r)$ where $r < 0$ and $0 < \varepsilon << |r|$.  
Let $\phi_{\varepsilon} = -(\log (-\varphi_{\varepsilon}))$ and let W be a connected open neighborhood of $p$ such that $W \subset V$.  
If $\varepsilon$ is sufficiently small,  
then $\mathrm{supp} (i\partial \overline{\partial}\varphi_{\varepsilon}) \subset V$ in $\Omega(r)$ and $\mathrm{supp}\, (i\partial \overline{\partial} \phi_{\varepsilon})^{n}\cap W \neq \emptyset$ by the continuity of the Monge-Amp\`ere measure (see \cite{Bed-Tay}, \cite{Kli}).  
Let $s < 0$ such that $W \subset \Omega(s)$.  
By taking $|r|$ and $\varepsilon$ are sufficiently small, 
we may assume that there exists $t<0$ such that $\Omega(s) \subset \Omega_{\varepsilon}(t):=\{z \in \Omega(r)\, |\, \varphi_{\varepsilon}(z) < t\}$ and 
$\overline{\Omega_{\varepsilon}(t)} \subset \Omega(r)$.  
Then there exists a holomorphic function $F_{t}$ on $\Omega_{\varepsilon}(t)$ such that $F_{t} = f$ on $\mathrm{supp}\, (i \partial \overline{\partial}\phi_{\varepsilon})^{n}$
by the same argument as the proof of Theorem~\ref{theorem:1}.  
Then $F_{t} = f$ on $W$ because $\mathrm{supp}\, (i\partial \overline{\partial}\phi_{\varepsilon})^{n}\cap W \neq \emptyset$.  
Let $\Omega(s)_{0}$ be a component of $\Omega(s)$ which contains $W$.  
It follows that $F_{t}|_{\Omega(s)_{0}}$ does not depend on $\varepsilon, r$ and $t$ by the identity theorem.  
By letting $s \to 0$, there exists the holomoprhic function $F$ on $\Omega$ such that $F = f$ on $W$.  
Since $V$ is connected, we have $F =f$ on $V$
\end{proof}
\begin{corollary}\label{corollary:2}
Let $n \geq 4$ and $\Omega$ be a pseudoconvex domain in $\mathbb{C}^{n}$.  
Let $\varphi$ be an exhaustive smooth plurisubharmonic function on $\Omega$.  
Let $V \subset \Omega$ be a connected open neighborhood of $\mathrm{supp}\, (i\partial \overline{\partial}\varphi)^{n-3}$.  
Then any holomorphic function on $V$ can be extended to the holomorphic function on $\Omega$.  
\end{corollary}
\begin{corollary}\label{corollary:3}
Let $n \geq 4$ and $\Omega$ be a pseudoconvex domain in $\mathbb{C}^{n}$.  
Let $\varphi$ be an exhaustive continuous plurisubharmonic function on $\Omega$.  
Let $V \subset \Omega$ be a connected open neighborhood of $\mathrm{supp}\, i\partial \overline{\partial}\varphi$.  
Then any holomorphic function on $V$ can be extended to the holomorphic function on $\Omega$.  
\end{corollary}
\begin{proof}[Proof of Corollary~\ref{corollary:2} and \ref{corollary:3}]
Let $r \in \mathbb{R}$.  
Then $\Omega(r) = \{z \in \Omega \, |\, \varphi(z) < r\}$ is a bounded hyperconvex domain and $\mathrm{supp}\, (i \partial \overline{\partial} (-\log(-\varphi)))^{n} = \mathrm{supp}\, (i \partial \overline{\partial} (-\log(r-\varphi)))^{n}$.  
The corollaries follow from the same arguments as the proofs of Theorem~\ref{theorem:1} and Corollary~\ref{corollary:1}.  
\end{proof}
Let $\Omega$ be a pseudoconvex domain in $\mathbb{C}^{n}$ ($n \geq 4$) 
and let $\varphi$ be an exhaustive continuous plurisubharmonic function on $\Omega$.  
Let $\Omega(r) = \{z \in \Omega \, |\, \varphi(z) < r\}$.  
Then $\max\{\varphi, r\}$ is an exhaustive continuous plurisubharmonic function which is pluriharmonic on $\Omega(r)$.  
Hence any holomorphic function on a connected open neighborhood of $\Omega \setminus \Omega(r)$ can be extended to the holomorphic function on $\Omega$.  
This is a special case of Hartogs extension theorem.  

\par\noindent{\scshape \small
Department of Mathematics, \\
Ochanomizu University,  \\
2-1-1 Otsuka, Bunkyo-ku, Tokyo (Japan) }
\par\noindent{\ttfamily chiba.yusaku@ocha.ac.jp}
\end{document}